\documentclass[reqno,11pt]{amsart}
\usepackage{amsmath,amsopn,amssymb,amsthm}
\usepackage{graphicx}
\newtheorem{theorem}{Theorem}
\newtheorem{lemma}[theorem]{Lemma} 
\newtheorem{cor}[theorem]{Corollary}
\newtheorem{prop}[theorem]{Proposition}

\newtheorem{defi}{Definition}

\DeclareMathOperator{\Pf}{Pf}

\def\N{\mathbb N}

\def\la{\lambda}
\DeclareMathOperator{\cross}{\textrm{cross}}
\DeclareMathOperator{\nest}{\textrm{nest}}

\begin{document}

\thispagestyle{empty}

%
%
%
%
%
\title{A one-parameter generalization of Pfaffians}
\author{T.~Eisenk\"olbl$^1$
  \and M.~Ishikawa$^2$
  \and J.~Zeng$^3$}



\address{$^{1,3}$ Universit\'e Lyon, Universit\'e Lyon 1, Institut
  Camille Jordan, CNRS UMR 5208, 43, blvd du 11 novembre 1918, 69622
  Villeurbanne cedex, France\newline
\indent  $^2$ Department of Mathematics, Faculty of Education,
  University of the Ryukyus, Nishihara, Okinawa, Japan\newline
\indent Email: {\tt eisenkoelbl@math.univ-lyon1.fr, ishikawa@edu.u-ryukyu.ac.jp,
  zeng@math.univ-lyon1.fr.} }


%
%
%
%
\begin{abstract}
%
In analogy to the definition of the lambda-determinant, we define a one-parameter deformation 
of the Dodgson condensation formula for Pfaffians. We prove that the resulting
rational function is a polynomial with weights given
by the crossings and nestings of perfect matchings and prove several
identities and closed-form evaluations.

%
\end{abstract}

\maketitle
%
%
%
%
%
\begin{section}{A recurrence relation for the generalized Pfaffians}
The $\lambda$-determinant \cite{RR} is defined by a generalization of
the Dodgson condensation formula (or the Desnanot-Jacobi formula) \cite{Br,Do,IO,IW}
\[
\det A\det A_{1,n}^{1,n}=\det A_{1}^{1}\det A_{n}^{n}-\det A_{1}^{n}\det A_{n}^{1}
\]
for a square matrix $A$ of size $n$,
where $A_{j_1,\dots,j_r}^{i_1,\dots,i_r}$ stands for the square matrix of size $n-r$
obtained from $A$ by deleting the rows $i_1,\dots,i_r$ and the columns $j_1,\dots,j_r$.
In \cite{Ze}, Zeilberger gave an amusing algorithmic proof for the identity, 
and recently two more combinatorial proofs appeared in \cite{Ay, Fu}.
In the paper \cite{RR}, Robbins and Rumsey used the $\lambda$-Dodgson condensation
\[
\det A\det A_{1,n}^{1,n}=\det A_{1}^{1}\det A_{n}^{n}-\lambda\det A_{1}^{n}\det A_{n}^{1}
\]
to define the $\lambda$-determinant
and established a weighted summation formula indexed by alternating sign matrices (ASMs) of size $n$,
which had a great impact on combinatorics, representation theory
and theoretical physics (see, for example, \cite{Br}) in  the last three decades.
Our motivation was to find a similar generalization of the Pfaffian Dodgson condensation
\cite{IO,IW,Kn}
\begin{multline}
\Pf(A)\Pf(A^{(1,2,2n-1,2n)})
=\Pf(A^{(1,2)})\Pf(A^{(2n-1,2n)})\\
-\Pf(A^{(1,2n-1)})\Pf(A^{(2,2n)})
+\Pf(A^{(1,2n)})\Pf(A^{(2,2n-1)}),
\label{eq:Pf-Dodgson}
\end{multline}
for a skew-symmetric matrix $A$ of size $2n$,
where we write $A^{(i_1,\dots,i_r)}$ for $A_{i_1,\dots,i_r}^{i_1,\dots,i_r}$.

We define the following one-parameter generalization:

\begin{defi}
The $\lambda$-Pfaffian $\Pf_{\la} A$ of a $2n \times 2n$ skew-symmetric
matrix $A$ is defined by the initial values
$$\Pf_{\la} ( () ) = 1,$$
$$\Pf_{\la}\left(\begin{pmatrix}0& a \\ -a& 0 \end{pmatrix} \right) =a,$$
and the recurrence
\begin{multline} \label{eq:lambdarec} 
\Pf_{\la} (A^{(1,2,2n-1,2n)})\Pf_{\la} (A)
=\Pf_{\la} (A^{(1,2)})\Pf_{\la} (A^{(2n-1,2n)})\\-\la \Pf_{\la} (A^{(1,2n-1)})\Pf_{\la} (A^{(2,2n)})+\la \Pf_{\la} (A^{(1,2n)})\Pf_{\la} (A^{(2,2n-1)}).
\end{multline}
for a skew-symmetric matrix $A$ of size $2n$,
where we write $A^{(i_1,\dots,i_r)}$ for $A_{i_1,\dots,i_r}^{i_1,\dots,i_r}$.
\end{defi}

In the special case $\lambda=1$, we recover the ordinary definition of a Pfaffian
so that we simply write $\Pf(A)$ for $\Pf_1(A)$.
Since a skew-symmetric matrix is completely determined by its upper triangle part,
we can safely denote  its ($\la$-)Pfaffian  by $\Pf_{\la}(A_{i,j})_{1\leq i<j\leq2n}$.

\par\smallskip
%


We need the following definitions of crossings and nestings of perfect
matchings to state Theorem~\ref{th:lambdarec}:

Let $\mathcal M(S)$ be the set of perfect matchings of a given finite set
$S \subset \N$. 
Any perfect matching can be written in the form $m=((m_1,m_2), \dots,
(m_{2n-1},m_{2n}))$, where $\{m_1,m_2,\dots, m_{2n}\} =S $, each pair
is sorted in increasing order and the first elements of each pair are
also sorted in increasing order, and let 
$\mathcal M([2n]) = \mathcal M (\left\{1,2, \dots, 2n\right\})$.
For a perfect matching $m$, the number $cross(m)$ of crossings is the number of pairs of pairs $(a,b)$ and $(c,d)$ in $m$ with $a < c < b < d$. Similarly, the number $nest(m)$ of nestings is the number of pairs of pairs $(a,b)$ and $(c,d)$ with $a < c < d < b$
 (see \cite{KZ}).


Now, we can give the explicit description of the coefficients of $\Pf_{\la}(A)$:
%
%
%
%
%
%
%

 \begin{theorem} \label{th:lambdarec}
The $\la$-Pfaffian of $A$  defined by the recurrence
relation~\eqref{eq:lambdarec} equals

$$\Pf_{\la}(A)=\sum_{m \in \mathcal M(2n)} (-1)^{\cross(m)}  \la^{\cross(m)+\nest(m)}
\prod_{i=1}^n A_{m_{2i-1}m_{2i}}.$$
\end{theorem}

Since every term in the expansion of the weighted sum is indexed
by matchings, we will give an explicit bijection between the pairs of
matchings corresponding to terms in the equation~\eqref{eq:lambdarec}.
In this sense, our proof gives another combinatorial proof of the Dodgson condensation formula 
by putting $\lambda=1$ and appealing to the fundamental relation between Pfaffians and determinants
(see \cite{IO,IW,Ste}), which is also the $\lambda=1$ case of Proposition~\ref{th:fundamental},
but our case is much more general.

The rest of this paper is organized as follows.
In Section~\ref{sec:proof} we give a purely combinatorial proof of Theorem~\ref{th:lambdarec}
which uses an involution.
In Section~3 we state and proof several properties of $\lambda$-Pfaffians,
which generalize several classical identities for Pfaffians.

%
%
%
%
%
%
\section{Proof of the main theorem}
\label{sec:proof}
In this section we give a combinatorial proof of Theorem~\ref{th:lambdarec}.
The following lemma will be needed to check that this bijection indeed
conserves the weight.
\begin{lemma} \label{lem:weightarc}
Let $m=((m_1,m_2), \dots, (m_{2n-1},m_{2n}))$ be a perfect matching of
$\{1,2,\dots, 2n\} $ with $m_{2i-1}<m_{2i}$. 
Then we have
$$\cross(m)+\nest(m) = \frac 12 \sum_{i=1}^n (m_{2i}-m_{2i-1}-1).$$
\end{lemma}

\begin{proof} 
The sum on the right-hand side counts for each pair in the
matching the number of integers that are between the two members of
the pair.

But each such interior point of a given pair $p$ has to be either a member of
a pair that crosses $p$ or that stays inside $p$ and is therefore
involved in exactly one set of crossing pairs or nesting pairs.

Conversely, each set of crossing pairs $(a,b)$ and $(c,d)$ with $a< c
< b < d$  is counted twice in the sum, once for the point $c$ inside
the pair $(a,b)$ and once for the point $b$ inside the pair $(c,d)$.

And each set of nesting pairs $(a,b)$ and $(c,d)$ with $a<c<d<b$ is
also counted twice in the sum, for the points $c$ and $d$ inside the pair
$(a,b)$.

Therefore, dividing by two gives exactly the weight $cross(m)+nest(m)$.
\end{proof}

Now, we are in the position to prove  the theorem, which is equivalent to
\begin{multline} \label{eq:lambdaeq}
\Pf_{\la} (A^{(1,2,2n-1,2n)})\Pf_{\la} (A)-\Pf_{\la} (A^{(1,2)})\Pf_{\la}
(A^{(2n-1,2n)})\\+\la \Pf_{\la} (A^{(1,2n-1)})\Pf_{\la} (A^{(2,2n)})-\la
\Pf_{\la} (A^{(1,2n)})\Pf_{\la} (A^{(2,2n-1)})=0.
\end{multline}

Each term in a product of two generalized Pfaffians is indexed by a
pair of matchings. More precisely, the terms in
Equation~\eqref{eq:lambdaeq} are indexed by the set

\begin{multline}{\label{eq:setdef}}
M= \mathcal M([2n]\setminus \left\{1,2,2n-1,2n \right\})
\times
\mathcal M(([2n]) \\
\overset{+}{\cup}
\mathcal M(([2n]\setminus\left\{1,2\right\})
\times
\mathcal M(([2n]\setminus\left\{2n-1,2n \right\} )\\
\overset{+}{\cup}
\mathcal M(([2n]\setminus\left\{1,2n-1 \right\} )
\times
\mathcal M(([2n]\setminus\left\{2,2n \right\} ) \\
\overset{+}{\cup}
\mathcal M(([2n]\setminus\left\{1,2n \right\} )
\times
\mathcal M(([2n]\setminus\left\{2,2n-1 \right\} ).
\end{multline}

We define an involution $\Phi$ on the set $M$ and show that it pairs
off terms that add to zero.

\begin{figure}[t]
\begin{center}
\leavevmode
\includegraphics[width=0.40\textwidth]{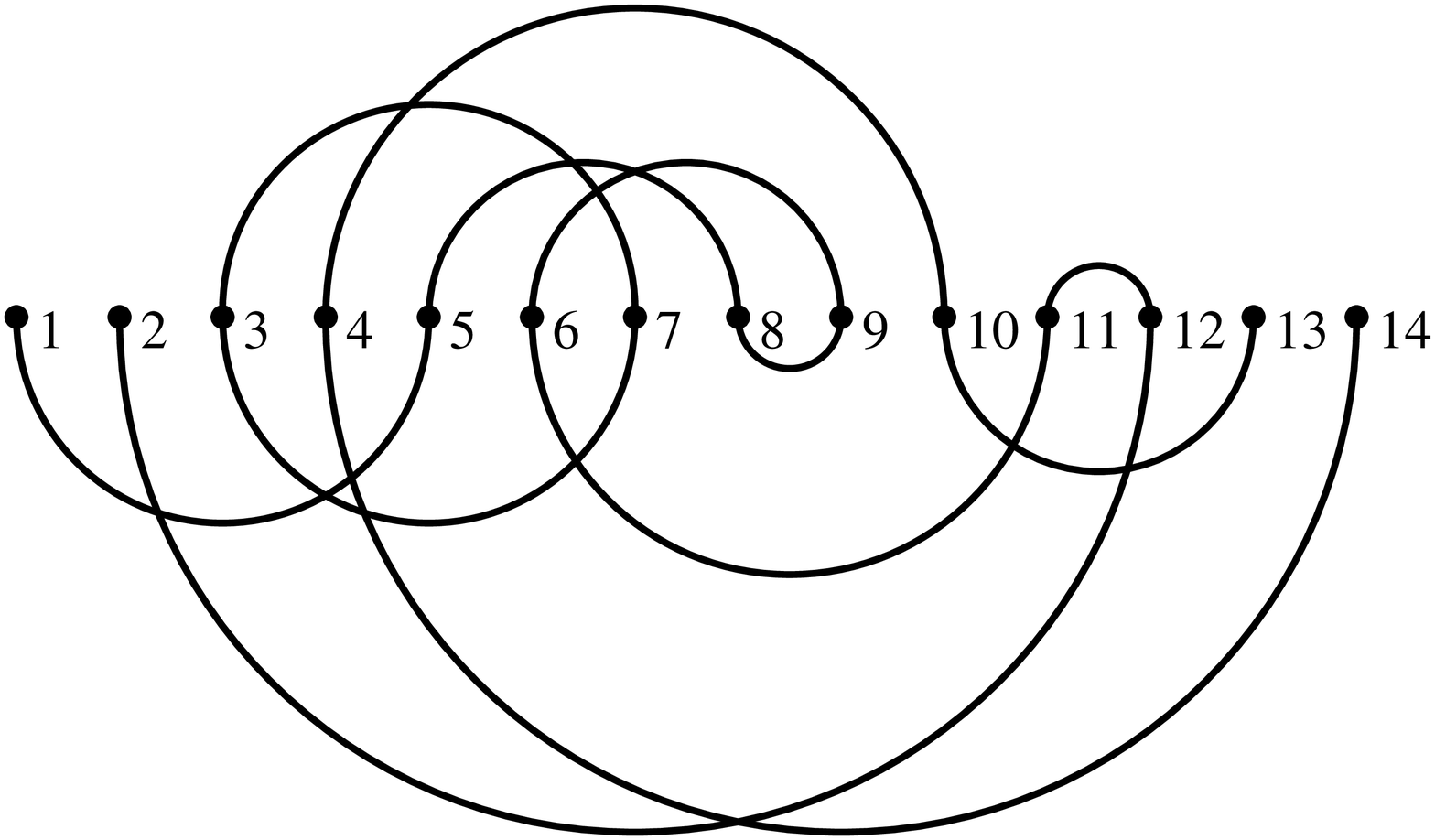}
\includegraphics[width=0.40\textwidth]{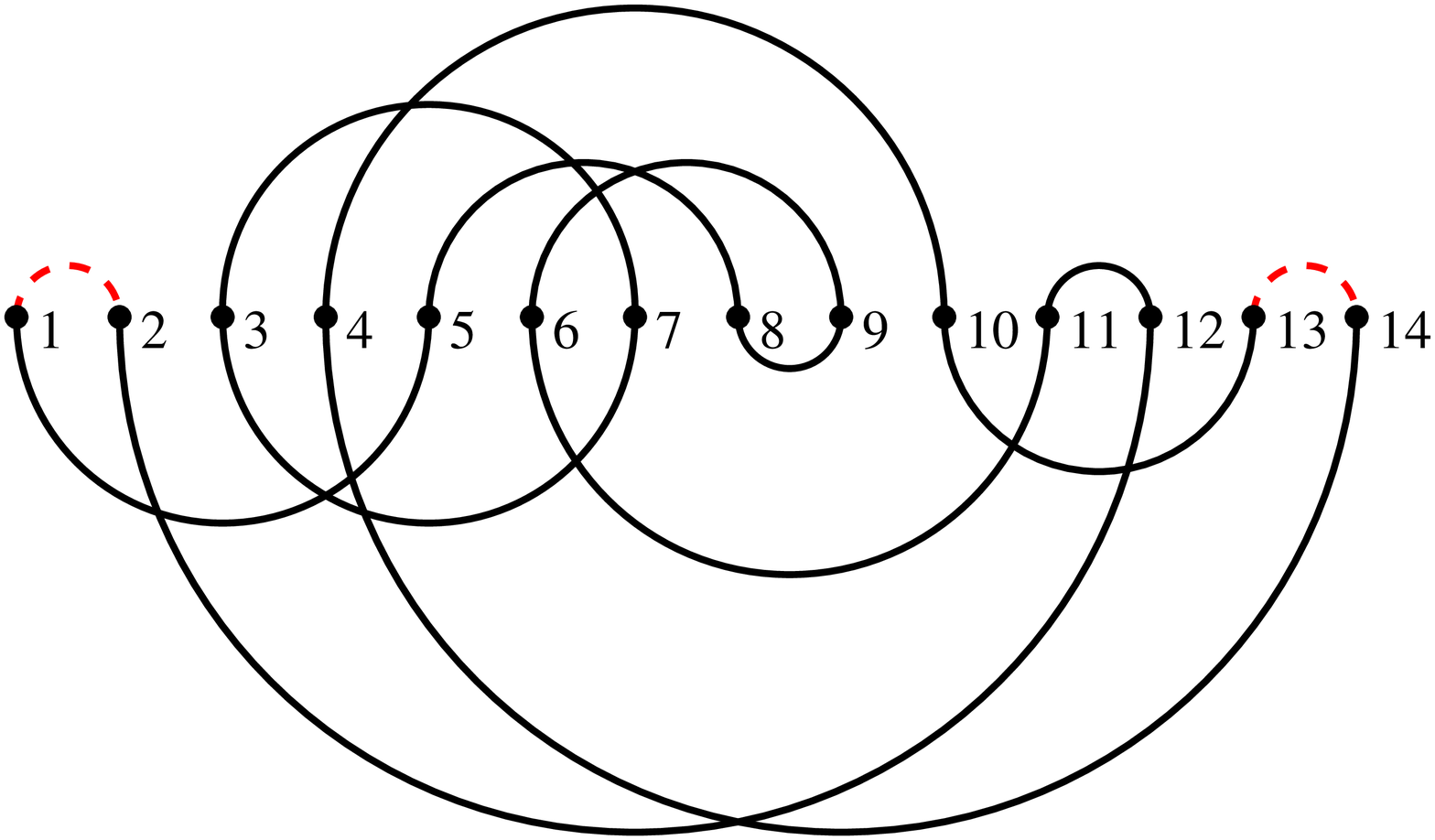}\\
\includegraphics[width=0.40\textwidth]{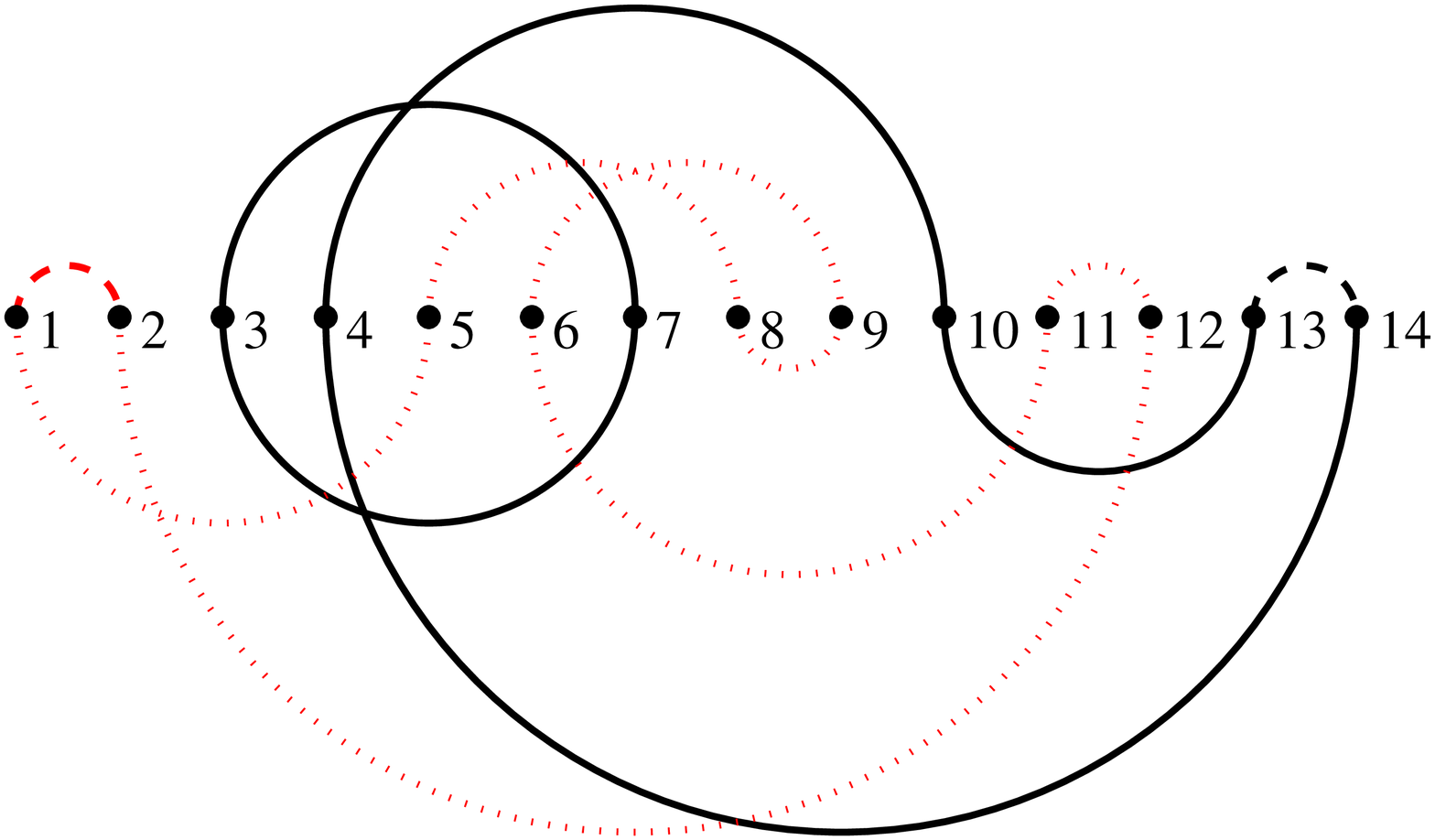}
\includegraphics[width=0.40\textwidth]{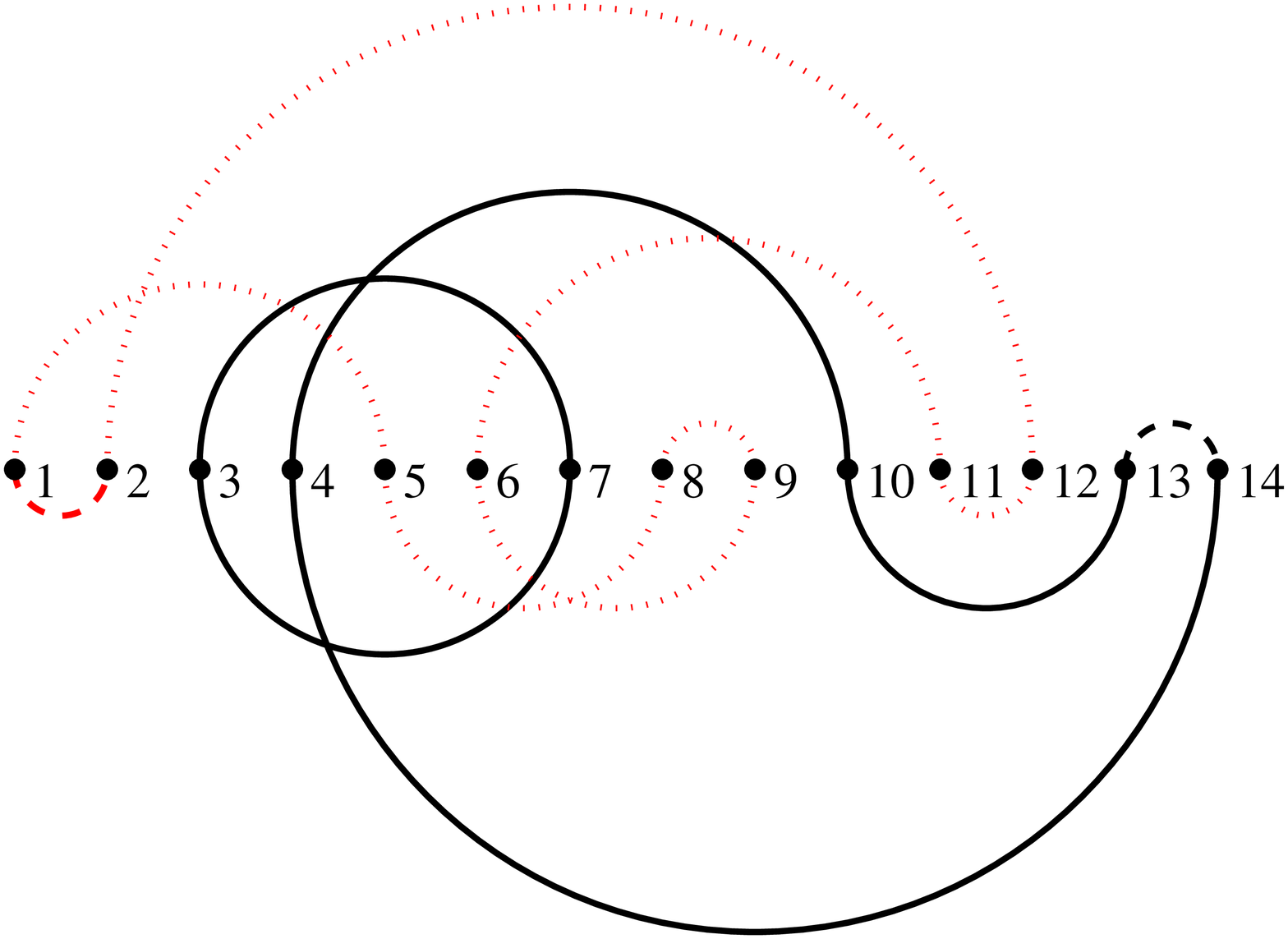}
\end{center}
\caption{The bijection on a pair of matchings}
\label{fig:bij}
\end{figure}


For a pair of matchings $(\pi, m) \in M$, we define
$\Phi(\pi,m)$ with the steps below.
Our running example will be $n=7$, $\pi =
((3,7),(4,10),(5,8),(6,9),(11,12)) \in \mathcal M([14] \setminus
\{1,2,13,14\})$ and $m =((1,5),(2,12),(3,7),(4,14), (6,11), (8,9),
(10,13)) \in \mathcal M ([14] )$.

{\em Step 1: Superpose the two matchings and add two additional pairs}

We represent the points $1,2,\dots, 2n$ on the horizontal line,
the pairs in $\pi$ as semicircles above the line and the pairs in $m$
as semicircles below the line (see the first image in
Figure~\ref{fig:bij}).

The superposition contains cycles and two paths starting and ending at
$1,2,2n-1,2n$. 

Note that the lengths of the two paths always have the same parity because the lengths of the cycles are necessarily even.

For a path running from $a$ and $b$, we would like to add the edge $(a,b)$. If the length of each path is odd, we can simply add the edges $(a,b)$ to the matching that does not yet contain $a$ and $b$.
This gives us two perfect matchings of the set $[2n]$
(see the two dashed lines in the second image in Figure~\ref{fig:bij})).

In our example, the two paths in the superpositions are
$1,5,8,9,6,11,12,2$ and $13,10,4,14$, so the additional edges $(1,2)$
and $(13,14)$ are added
to $\pi$ and we
get the two matchings
$$\pi'=((1,2),(3,7),(4,10),(5,8),(6,9),(11,12),(13,14))$$ and
$$m=((1,5),(2,12),(3,7),(4,14),(6,11),(8,9),(10,13)).$$

However, if the length of each path is even, then we add an additional dummy vertex $b'$ for each edge $(a,b)$ and we add the edges $(a,b')$ to the matchings that does not yet contain $a$ and we add the edges $(b',b)$ to the matching that does not yet contain $b$.
(See Figure~\ref{fig:bijright} for an 
example.)

{\em Step 2: Flip the cycle containing 1}

With the additional edges, we have two perfect matchings and the
superposition is a set of cycles.  Now,
we choose the cycle that contains 1.

In our example, this is the cycle $(1,5,8,9,6,11,12,2)$ (see the red
cycle in the third image in Figure~\ref{fig:bij}).

In this cycle only, we change the roles of the two matchings. This
corresponds to a reflection with respect to the horizontal axis in the
graphical representation (see the fourth image in
Figure~\ref{fig:bij}).

Note that the two additional edges are in different cycles by
construction, so the additional edge containing 1 is moved
from the matching $\pi'$ to the matching $m$, but not the other.

{\em Step 3: Remove the additional edges and switch the two perfect
matchings}

Now, we can remove the additional edges again, switch the two matchings and
find two new matchings $\tilde\pi$ and $\tilde m$.

If a new vertex was added in Step 1, this vertex is now removed along
with the additional edges.

In our example, this gives
$$\tilde\pi=((3,7),(4,14),(5,8),(6,9),(10,13),(11,12))$$ and
$$\tilde m=((1,5),(2,12),(3,7),(4,10),(6,11),(8,9)).$$

{\bf Claim 1: $\Phi$ is an involution.}
Since the cycle containing 1 remains the cycle containing 1, this just
means that we repeat the same reflection twice if we apply $\Phi$ twice.

{\bf Claim 2: Switching the two perfect matchings preserves the weight}
Each term in the product of two generalized Pfaffians is of the form
$w(\pi)w(m)$, so interchanging the two matchings does not change the
product.

{\bf Claim 3: Reflecting a cycle preserves the weight}

The number of crossings corresponds exactly to the intersections of
arcs in the graphical representation. It is easy to see that two
circles intersect an even number of times because a circles has to
enter and leave the interior of the other circle the same number of
times.

Since reflecting does not change the number of crossings of a cycle
with itself,  we already see that the parity of the number of
crossings is conserved, so $(-1)^{cross(\pi)+cross(m)}$ does not change.

We have already seen in Lemma~{\ref{lem:weightarc}} that the sum
$cross(m)+nest(m)$ only depends on the number of interior points of an
arc. Since the reflection does not change the total set of arcs, the
weight $cross(\pi)+nest(\pi)+cross(m)+nest(m)$ remains unchanged.

Note that in some cases, adding the two edges will create edges that
share a point (see Figure~\ref{fig:bijrightcross} for an example). But the
arguments in Lemma~\ref{lem:weightarc} remain true if we count nested
arcs that share a point as half a nesting.

\begin{lemma}
Adding and removing the two edges will change the weight in
  the way that will make the two terms cancel.
\end{lemma}
\begin{proof}
For this final part,
we have to check the six possible cases for the
partial matchings and additional edges that the terms in
Equation~\eqref{eq:lambdaeq} cancel.

In all pictures, the dashed lines are the additional edges and the dotted (red) lines are the lines of the cycle containing 1.

{\em Case 1: Additional edges $(1,2)$ and $(2n-1,2n)$.Between
  $\mathcal M([2n]\setminus\left\{1,2,2n-1,2n\right\}\times
\mathcal M([2n])$ and
  $\mathcal M([2n]\setminus\left\{1,2\right\}\times
\mathcal M([2n]\setminus\left\{2n-1,2n\right\})$
   }

This is the case of Figure~\ref{fig:bij}.
The additional edges cannot be involved in any crossing
or nesting, so adding or removing them does not change the weight at
all.

The two corresponding terms in Equation~\eqref{eq:lambdaeq} have
coefficients $+1$ and $-1$, so they cancel.

\begin{figure}[t]
\begin{center}
\leavevmode
\includegraphics[width=0.40\textwidth]{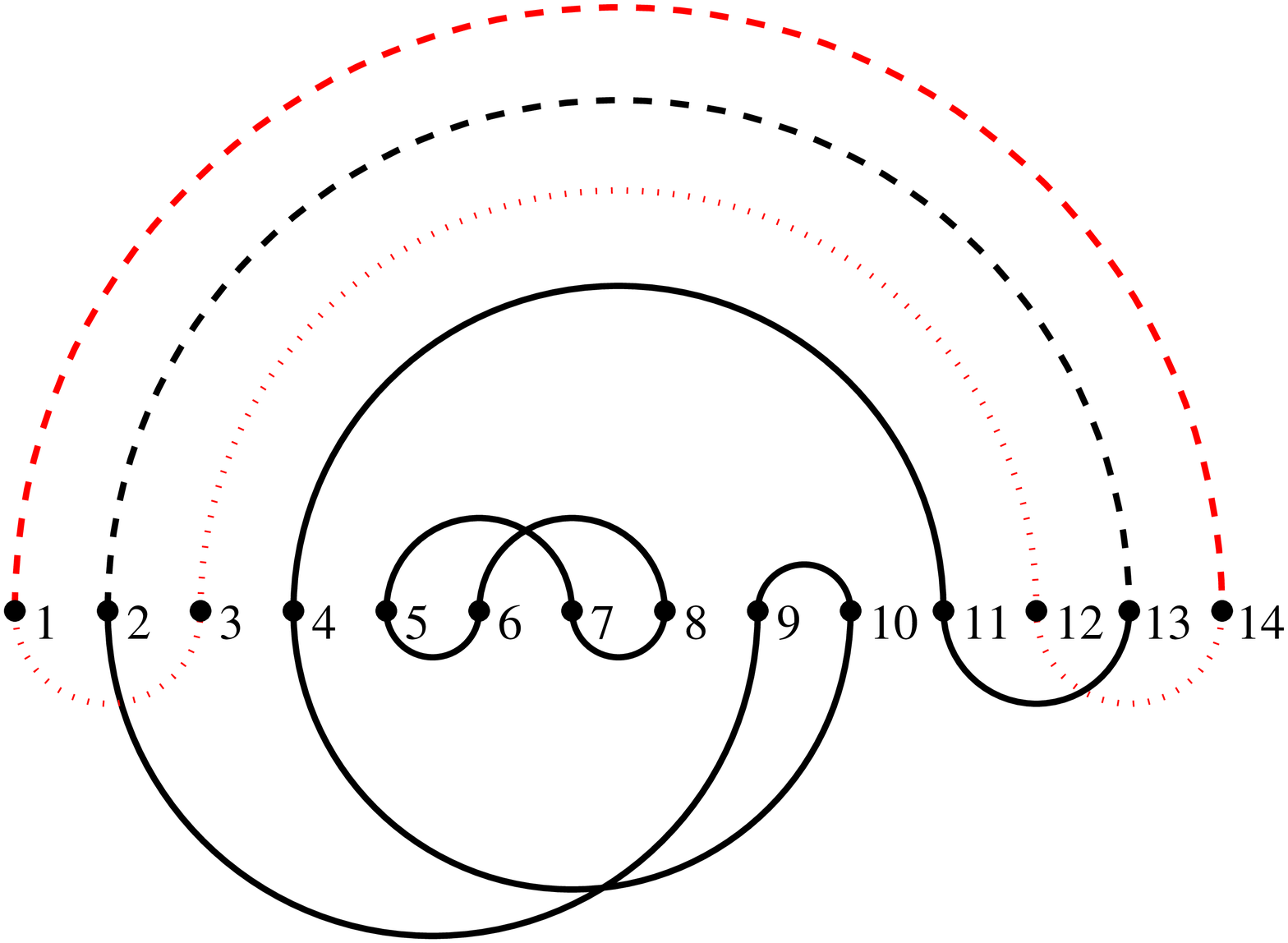}
\includegraphics[width=0.40\textwidth]{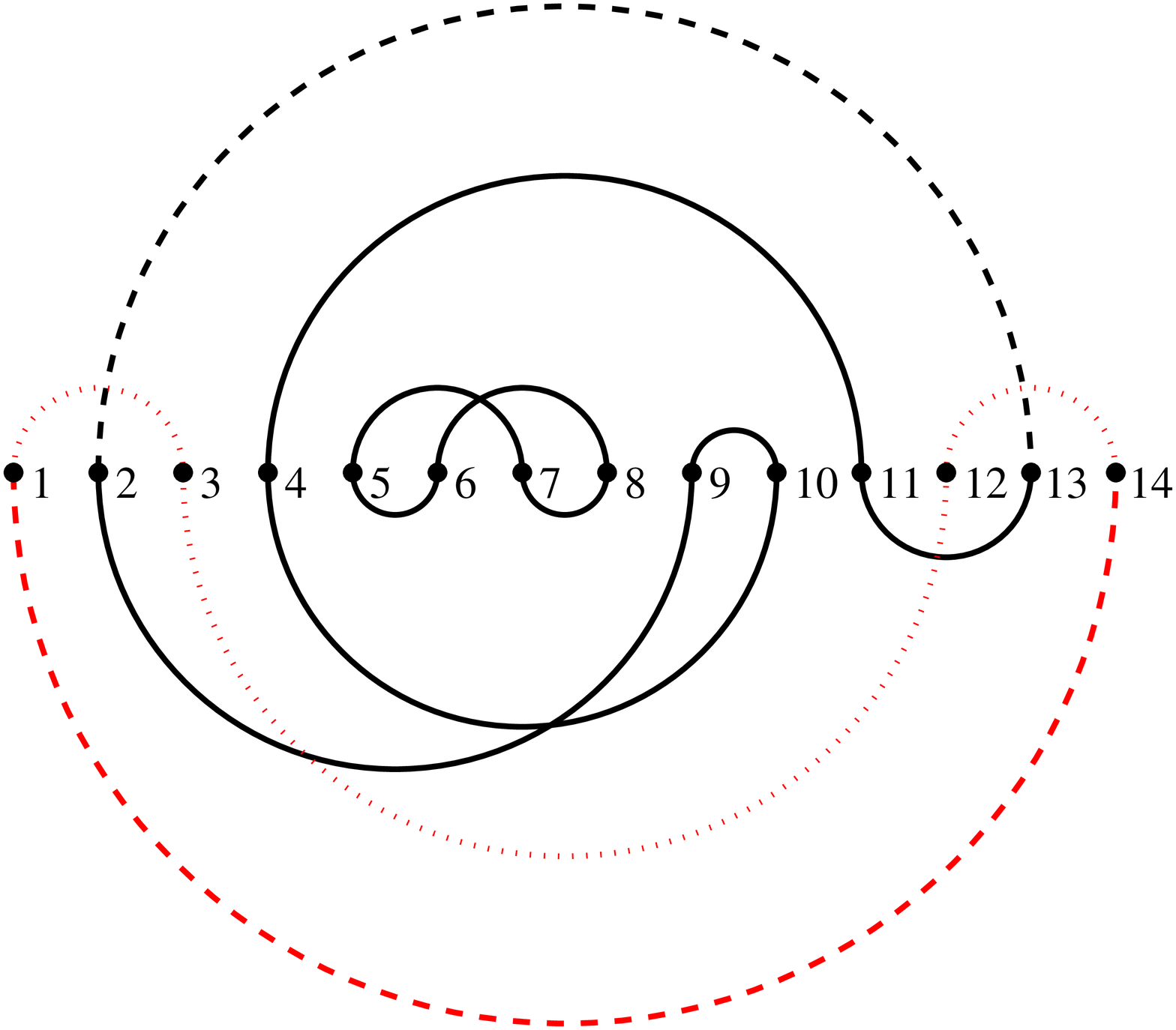}
\end{center}
\caption{The bijection on a pair of ``nesting'' matchings}
\label{fig:bijnest}
\end{figure}

{\em Case 2: Additional edges $(1,2n)$ and $(2,2n-1)$.Between
  $\mathcal M([2n]\setminus\left\{1,2,2n-1,2n\right\}\times
\mathcal M([2n])$ and
  $\mathcal M([2n]\setminus\left\{1,2n\right\}\times
\mathcal M([2n]\setminus\left\{2,2n-1\right\})$
   }

This is the case of Figure~\ref{fig:bijnest}.

Adding the edges will generate no crossings and $(n-1)+(n-2)$ nestings
(this is a consequence of the arguments in the proof of
Lemma~\ref{lem:weightarc}).

Removing them will remove no crossings and $(n-1)+(n-1)$ nestings if
$(1,2n)$ is a pair in matching $\pi$ and it will remove 2 crossings
and $(n-1)+(n-3)$ nestings if it is not.

So, in total the weight change is $\la^{(2n-3)-(2n-2)}=\la^{-1}$.

The two corresponding terms in Equation~\eqref{eq:lambdaeq} have
coefficients $+1$ and $-\la$, so they cancel.

\begin{figure}[t]
\begin{center}
\leavevmode
\includegraphics[width=0.40\textwidth]{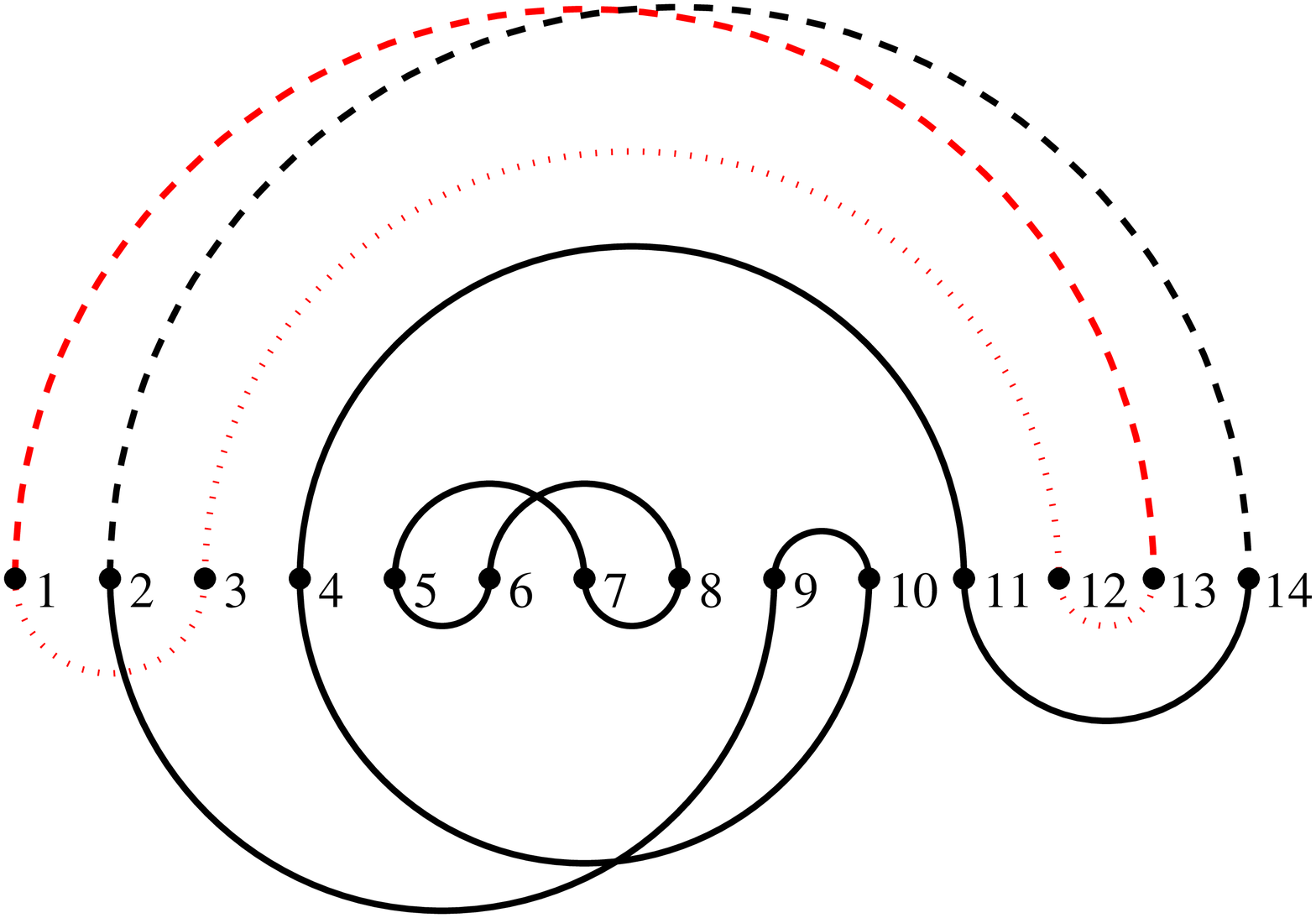}
\includegraphics[width=0.40\textwidth]{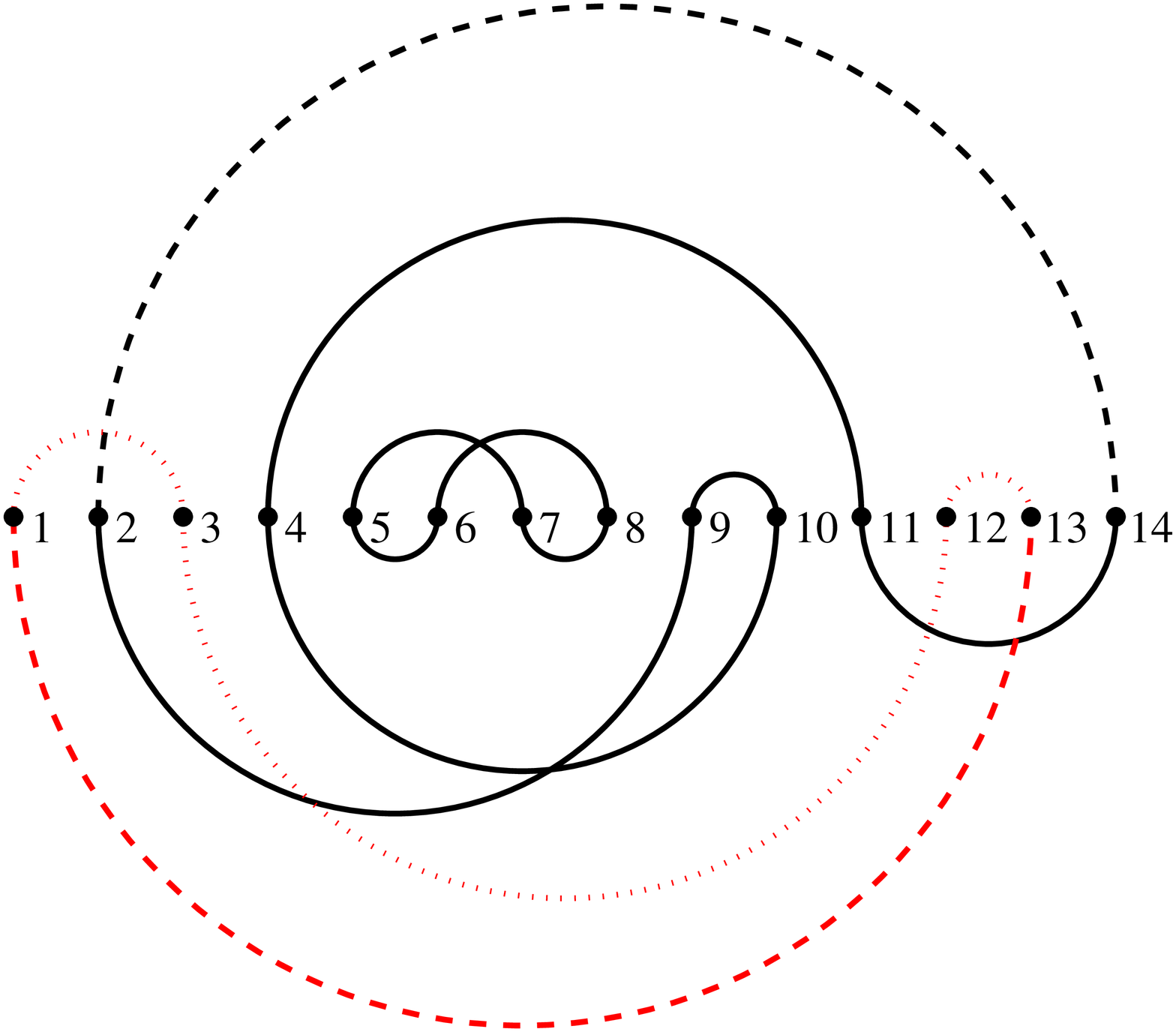}
\end{center}
\caption{The bijection on a pair of ``crossing'' matchings}
\label{fig:bijcross}
\end{figure}

{\em Case 3: Additional edges $(1,2n-1)$ and $(2,2n)$.Between
  $\mathcal M([2n]\setminus\left\{1,2,2n-1,2n\right\}\times
\mathcal M([2n])$ and
  $\mathcal M([2n]\setminus\left\{1,2n-1\right\}\times
\mathcal M([2n]\setminus\left\{2,2n\right\})$
   }

This is the case of Figure~\ref{fig:bijcross}.

Adding the edges will generate one crossing and $2n-4$ nestings.

Removing them will remove two crossings and $(n-2)+(n-2)$ nestings.

So, in total the weight change is $-\la^{(2n-3)-(2n-2)}=-\la^{-1}$.

The two corresponding terms in Equation~\eqref{eq:lambdaeq} have
coefficients $+1$ and $\la$, so they cancel again.

\begin{figure}[t]
\begin{center}
\leavevmode
\includegraphics[width=0.40\textwidth]{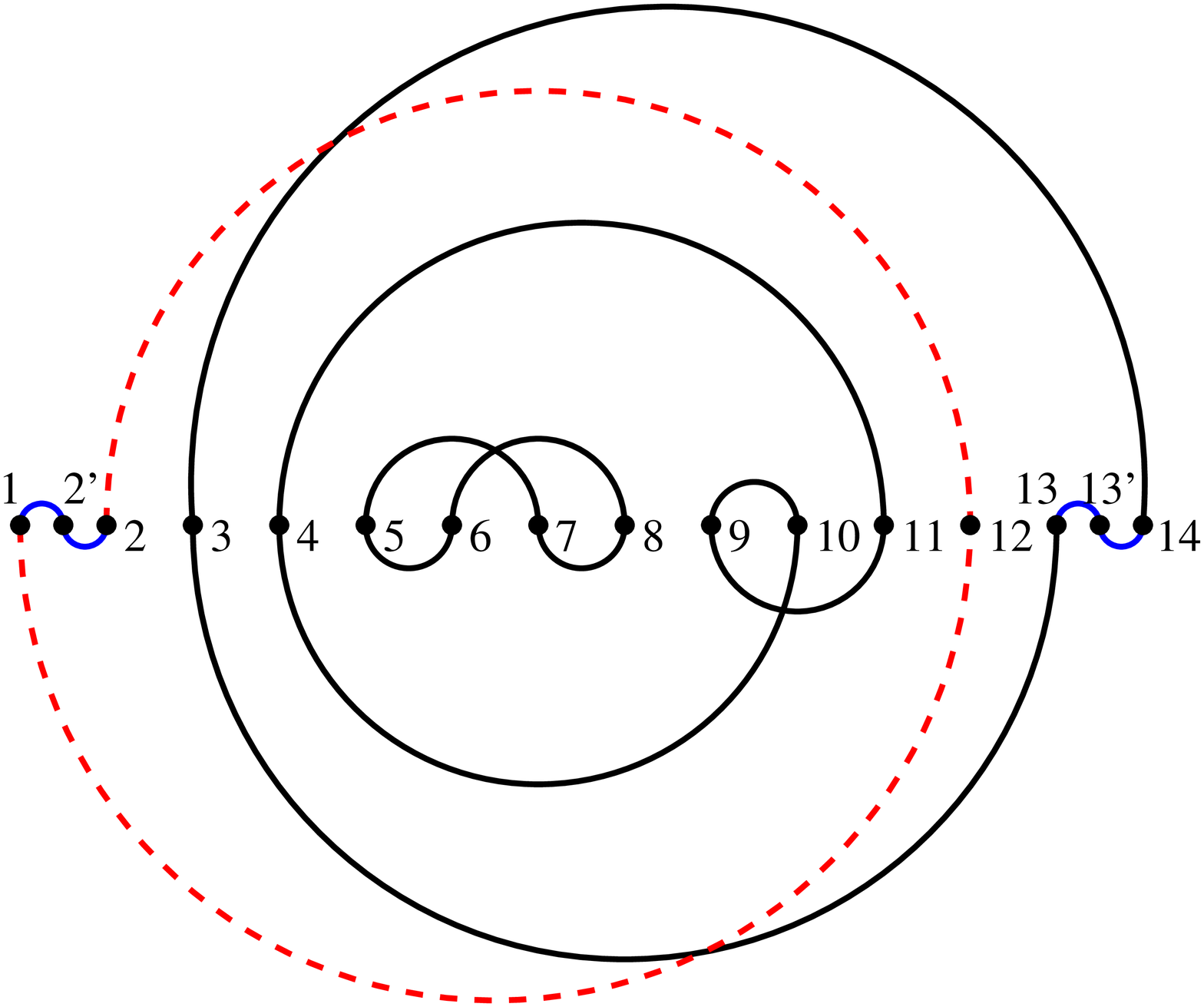}
\includegraphics[width=0.40\textwidth]{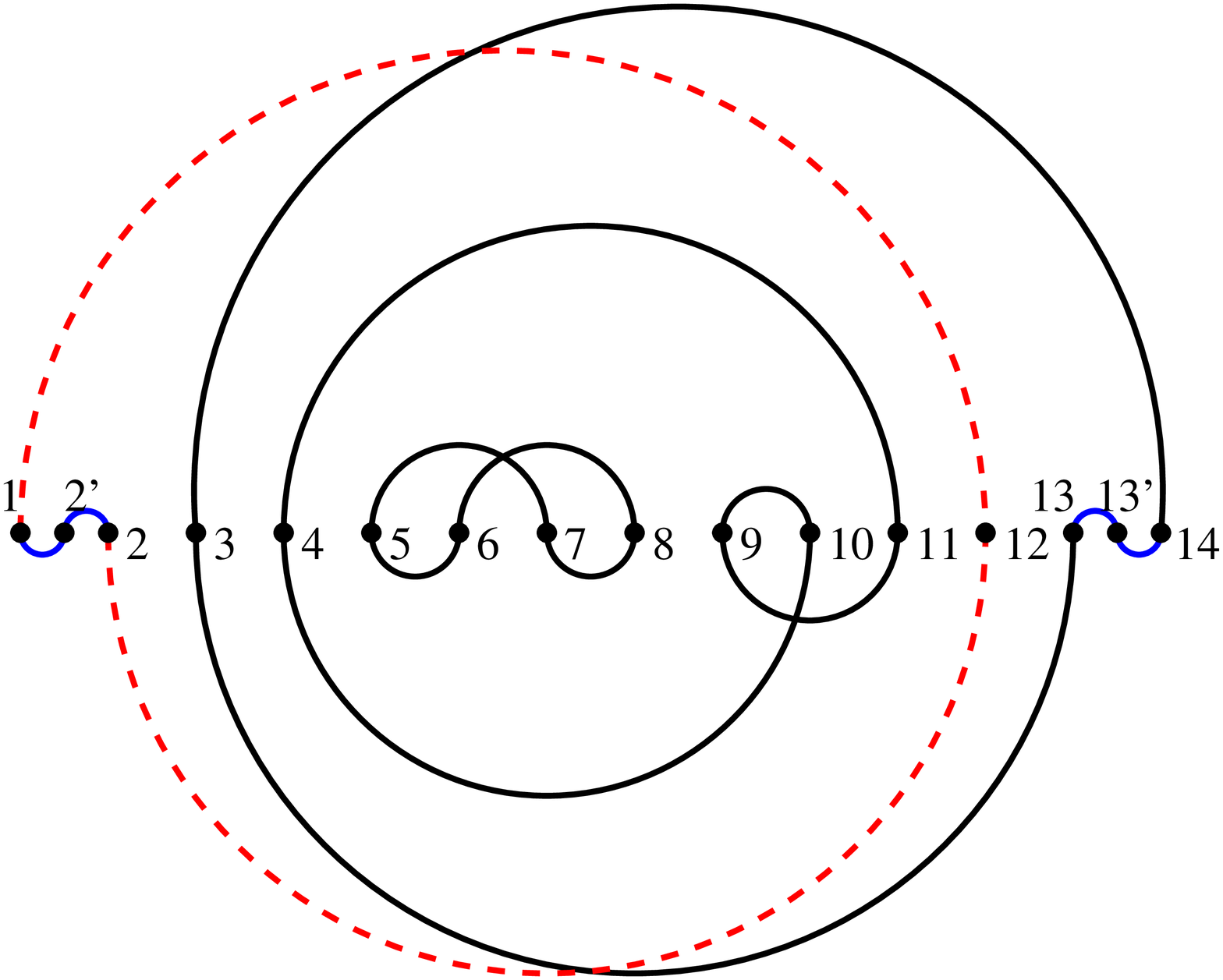}
\end{center}
\caption{The bijection on a pair of right-hand side matchings}
\label{fig:bijright}
\end{figure}

{\em Case 4: Additional edges $(1,2)$ and $(2n-1,2n)$.Between
  $\mathcal M([2n]\setminus\left\{1,2n-1\right\})\times
\mathcal M([2n]\setminus\left\{2,2n\right\})$ and
  $\mathcal M([2n]\setminus\left\{1,2n\right\})\times
\mathcal M([2n]\setminus\left\{2,2n-1\right\})$
   }
This is the case of Figure~\ref{fig:bijright}.

We have to split two vertices and add four edges to get two perfect
matchings.

Adding the edges adds two nestings and no crossings
and removing the edges removes two nestings, so the weight remains
unchanged.

The two corresponding terms in Equation~\eqref{eq:lambdaeq} have
coefficients $\la$ and $-\la$, so they cancel again.

\begin{figure}[t]
\begin{center}
\leavevmode
\includegraphics[width=0.40\textwidth]{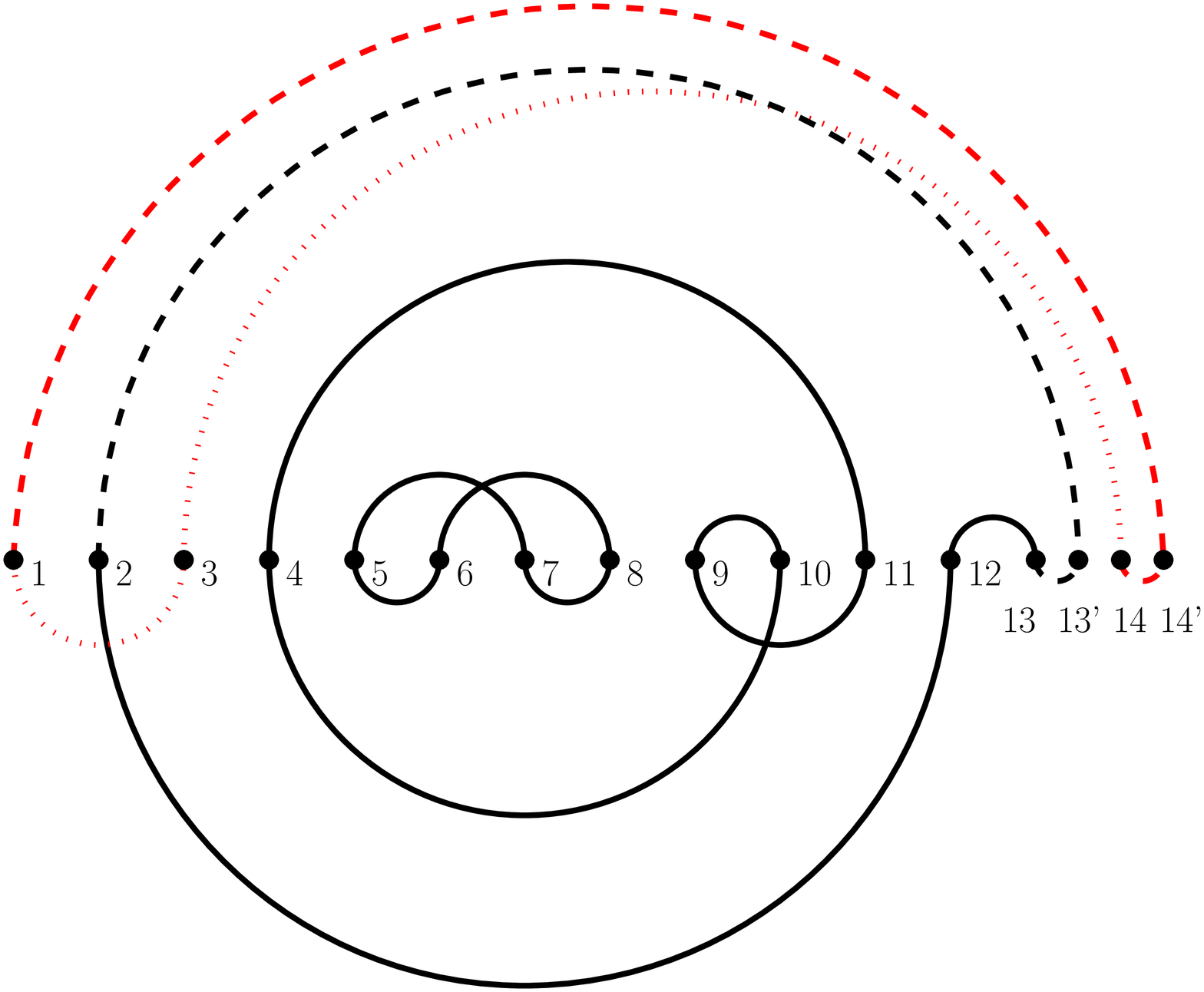}
\includegraphics[width=0.40\textwidth]{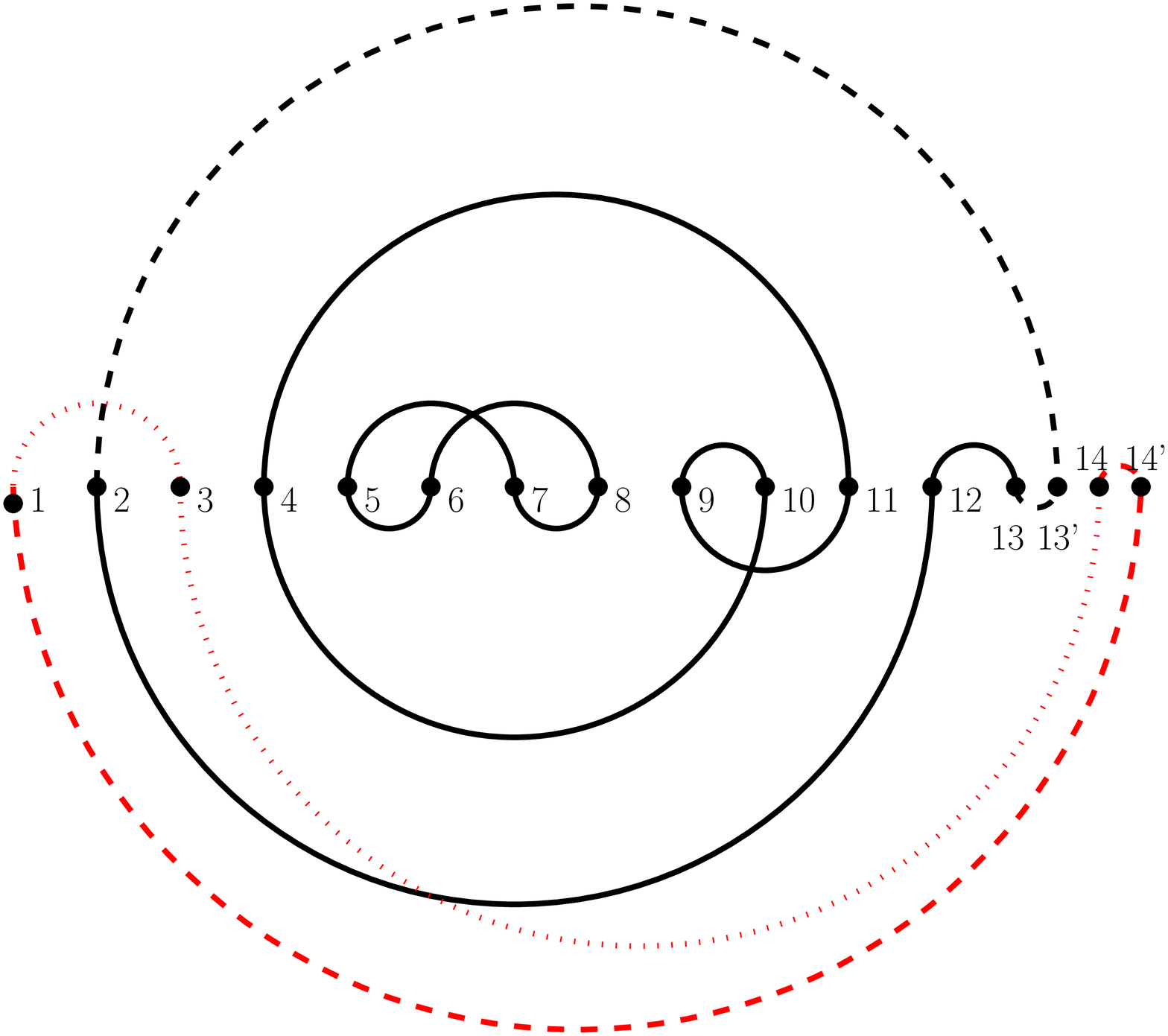}
\end{center}
\caption{The bijection on a pair of right-hand side ``nesting" matchings}
\label{fig:bijrightnest}
\end{figure}

{\em Case 5: Additional edges $(1,2n)$ and $(2,2n-1)$.Between
  $\mathcal M([2n]\setminus\left\{1,2\right\})\times
\mathcal M([2n]\setminus\left\{2n-1,2n\right\})$ and
  $\mathcal M([2n]\setminus\left\{1,2n-1\right\})\times
\mathcal M([2n]\setminus\left\{2,2n\right\})$
   }

This is the case of Figure~\ref{fig:bijrightnest}.

Again, two vertices have to be split.

Adding the edges will generate 1 crossing and $n+(n-2)$ nestings.

Removing them will remove 1 crossing and $(n-2)+(n)+1$
nestings.

So, in total the weight change is $\la^{(2n-1)-(2n)}=\la^{-1}$.

The two corresponding terms in Equation~\eqref{eq:lambdaeq} have
coefficients $-1$ and $\la$, so they cancel again.

\begin{figure}[t]
\begin{center}
\leavevmode
\includegraphics[width=0.40\textwidth]{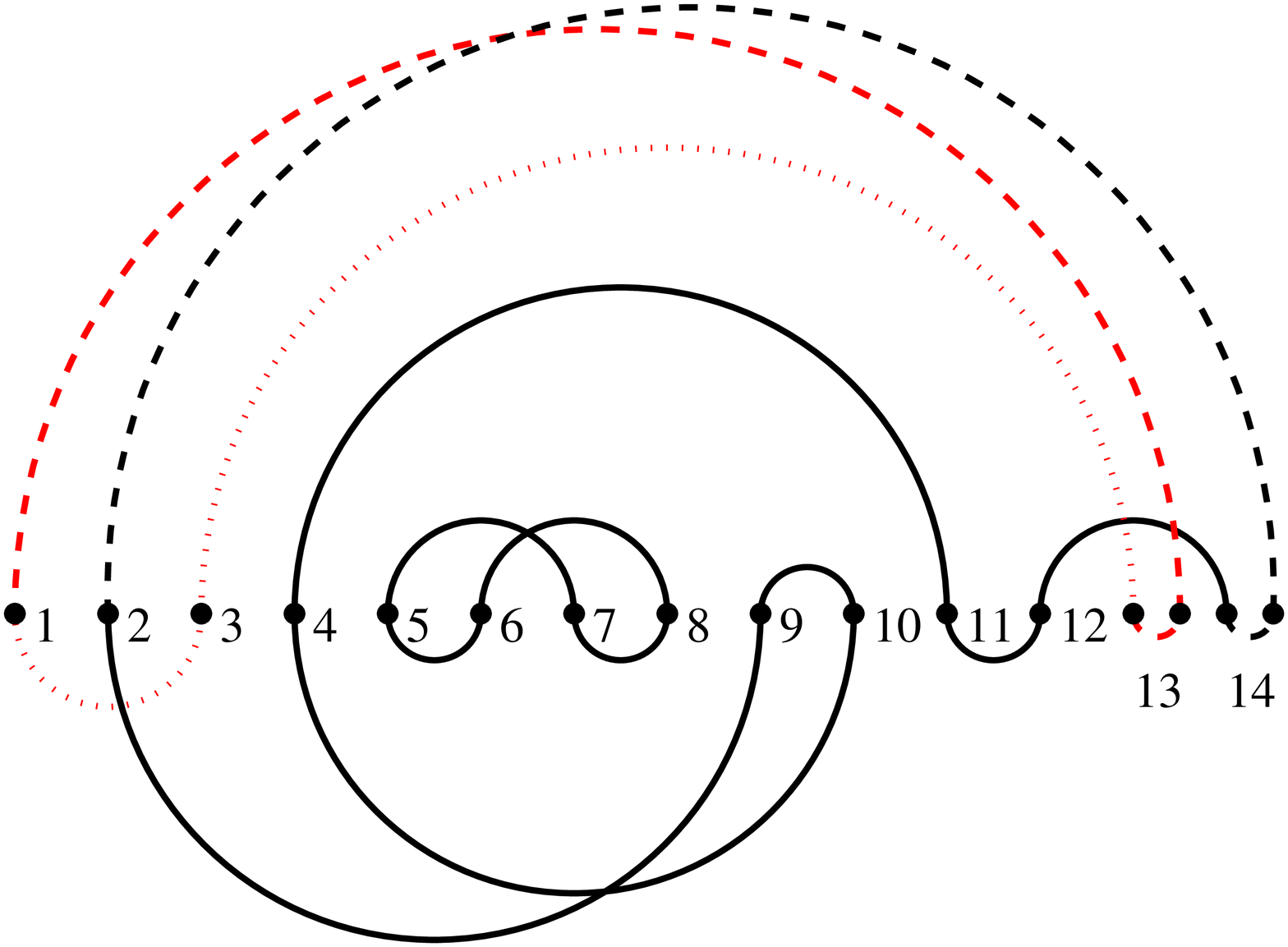}
\includegraphics[width=0.40\textwidth]{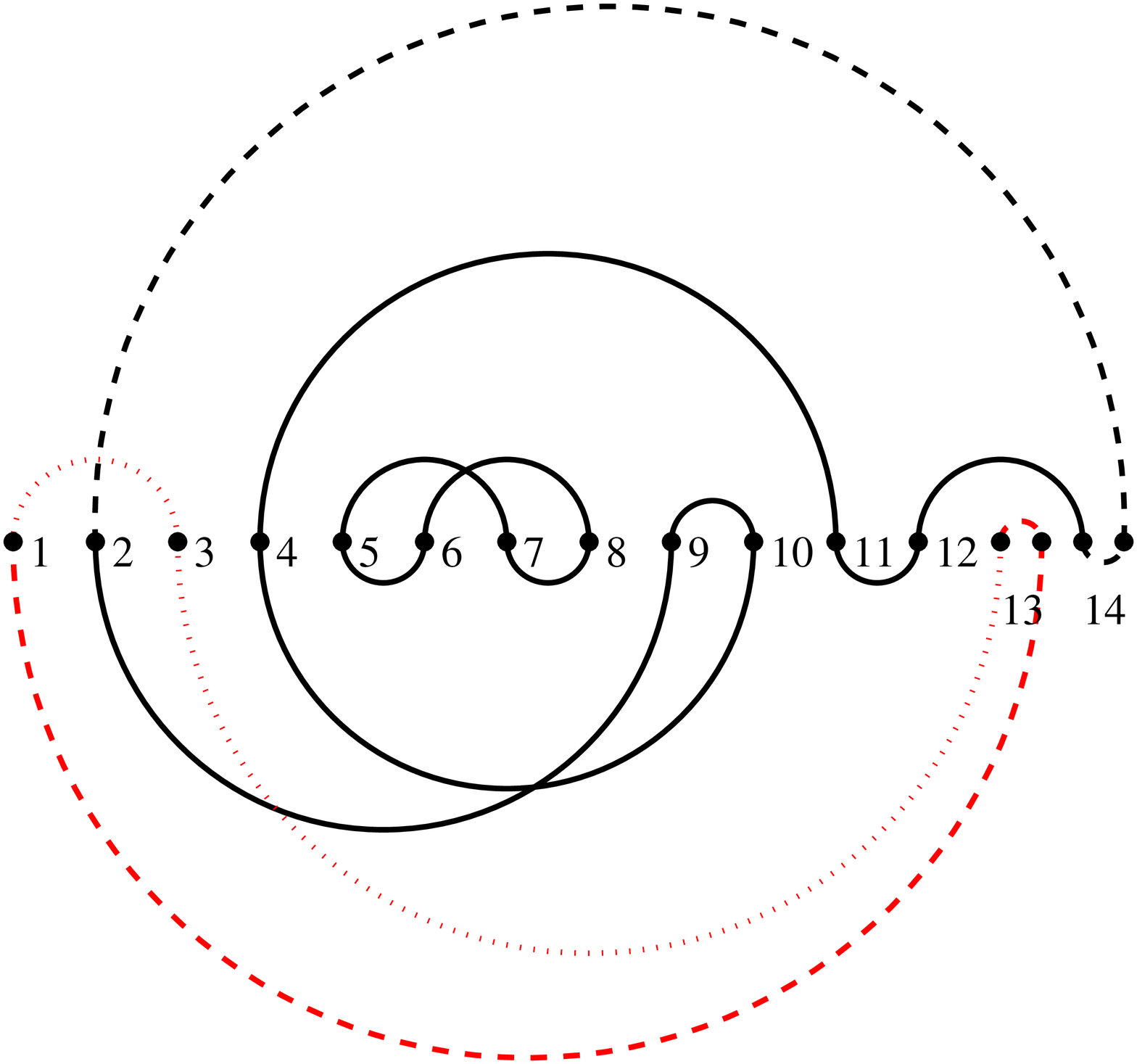}
\end{center}
\caption{The bijection on a pair of right-hand side ``crossing" matchings}
\label{fig:bijrightcross}
\end{figure}

{\em Case 6: Additional edges $(1,2n-1)$ and $(2,2n)$.Between
  $\mathcal M([2n]\setminus\left\{1,2\right\})\times
\mathcal M([2n]\setminus\left\{2n-1,2n\right\})$ and
  $\mathcal M([2n]\setminus\left\{1,2n\right\})\times
\mathcal M([2n]\setminus\left\{2,2n-1\right\})$
   }

This is the case of Figure~\ref{fig:bijrightcross}.

Again, we split two vertices and add four edges.

Adding the edges will generate 2 crossings and $(n-1)+(n-2)$ nestings.

Removing them will remove 1 crossing and $(n-1)+(n-1)+1$ nestings.

So, in total the weight change is $-\la^{(2n-1)-(2n)}=-\la^{-1}$.

The two corresponding terms in Equation~\eqref{eq:lambdaeq} have
coefficients $-1$ and $-\la$, so they cancel again.
\end{proof}


\end{section}

%
\begin{section}{Properties of $\lambda$-Pfaffians}
\label{properties}
In this section we investigate some properties of $\lambda$-Pfaffians.
We first  state and prove the following basic proposition
that gives the relation between $\lambda$-Pfaffians and determinants,
which generalize the classical identity
(see \cite{IO,IW,Ste}).
%

%
%
\begin{prop}
\label{th:fundamental}
Let $n$ be a nonnegative integer, and let $A$ be a square matrix of size $n$.
Then we have
\begin{equation}
\Pf_\lambda\begin{pmatrix}
O_{n}&A\\
-A^T&O_{n}
\end{pmatrix}
=(-\lambda)^{n(n-1)/2}\det A,
\end{equation}
where $O_n$ is the zero matrix of size $n$.
and $A^T$ stands for the transpose of $A$.
\end{prop}

\begin{proof}
Each number in $\{1,2,\dots,n\}$ has to be matched with a number in $\{n+1,n+2,\dots,2n\}$ to contribute a non-zero term to the $\lambda$-Pfaffian. Any two pairs in the matching are either crossing or nesting because both start points come before both end points. So, every term has a weight of $\la^{n(n-1)/2}$. After taking out this common factor, we obtain exactly the ordinary Pfaffian which is well-known to be $(-1)^{n(n-1)/2} \det A$ for this matrix.
\end{proof}

There are several basic properties of Pfaffians.
For example, 
the following identity is well known for $\lambda=1$: 
\begin{equation}
\Pf_{\lambda}(1)_{1\leq i<j\leq 2n}=1.
\label{eq:all1}
\end{equation}
This identity is obvious from the recurrence \eqref{eq:lambdarec} by induction.
Note that \eqref{eq:all1} also follows from the following  known continued fraction expansion~\cite{KZ}:
\begin{align*}
&
\sum_{n=0}^\infty\biggl(
\sum_{m \in \mathcal M(2n)} p^{cross(m)}q^{nest(m)}
\biggr)t^n
=\frac{1}{1-\displaystyle
\frac{[1]_{p,q}\,t}{1-\displaystyle
\frac{[2]_{p,q}\,t}{1-\displaystyle
\frac{[3]_{p,q}\,t}{\ddots}}}},
\end{align*}
because, if we put $p=-\lambda$ and $q=\lambda$ in this continued fraction,
it becomes $\frac1{1-t}$.
Here  $[k]_{p,q}=\frac{p^k-q^k}{p-q}$.
In fact, we obtain a more general formula as follows.
\begin{prop}
\label{th:xiyj}
Let $x_i$ and $y_i$ ($i=1,2,\dots$) be indeterminates, 
and let  $n$ be a positive integer.
Then we have
\begin{equation}
\Pf_{\lambda}\biggl(x_iy_j\biggr)_{1\leq i<j\leq2n} 
=\prod_{i=1}^{n}x_{2i-1}y_{2i}.
\label{eq:xiyj}
\end{equation}
\end{prop}
Proposition~\ref{th:xiyj} can be derived directly from \eqref{eq:lambdarec},
but here we present another combinatorial proof using an involution.

\begin{proof}
As usual, we write a perfect matching $m$ as $\{\{m_1,m_2\}, \{m_3,m_4\},\dots,
\{m_{2n-1},m_{2n}\}\}$ with $m_{2i-1} <m_{2i}$ for all $i$ and $m_1 <
m_3 < m_5 < \dots < m_{2n-1}$.

By the definition of the $\la$-Pfaffian, the left-hand side is a sum
over all perfect matchings $m$  with the sign $(-1)^{\text{cross(m)}}$
and the weight $\la^{cross(m) + nest(m)} \prod_{i=1}^n x_{m_{2i-1}}
y_{m_{2i}}$.

The expression on the right-hand side is exactly the summand
corresponding to the trivial matching $\{\{1,2\},\{3,4\},\dots
\{2n-1,2n\}\}$. We will describe a weight-preserving sign-reversing
involution to show that all other terms in the Pfaffian disappear.

Let $m$ be a matching different than the trivial one. Let $2i$ be the
first index $k$ such that $m_k \not= k$. Then $m_{2i-1} = 2i-1$ and
$m_{2i+1} = 2i$.
Now define a new matching by interchanging the values of $m_{2i}=a$ and $m_{2i+2}=b$.
This action is clearly an involution on the non-trivial matchings.

Any crossing or nesting of a pair not involving $2i-1$ with the pair $(2i-1,a)$
is also a crossing or nesting with the pair $(2i, a)$. This shows that the only
affected crossings or nestings are the ones between $(2i-1,a)$ and
$(2i,b)$. But in this case, we will just replace one crossing by a
matching or vice versa.

Therefore, the product of $x$'s and $y$'s and the sum $cross(m) + nest(m)$
is clearly invariant, while the number of crossings changes by one which
reverses the sign as desired.
\end{proof}

%
%
%
%
%
%
A Vandermonde--type identity for the Pfaffian is as follows:
\[
\Pf\biggl(\frac{(x_i^n-x_j^n)^2}{x_i-x_j}\biggr)_{1\leq i<j\leq2n}
=\prod_{1\leq i<j\leq2n}(x_i-x_j).
\]
We have not found a generalization of this formula for our $\lambda$-Pfaffian,
but we present the following theorem,
which generalizes a special case of this identity.
%
%
%
\begin{theorem} \label{xydiff}
For $n\ge 1$ we have
$$\Pf_{\la} (x_i -y_j)_{1\leq i<j\leq2n} = \prod_{k=1}^{n-1}(1-\la \sigma_{2k})
\prod_{i=1}^n (x_{2i-1}-y_{2i}),$$
where $\sigma_k$ is the (linear and multiplicative) operator with
$\sigma_k(x_{k+1})=y_{k+1}$, $\sigma_k(y_k)=x_k$, $\sigma_k(x_i)=x_i$ for $i\not=k+1$
and $\sigma_k(y_i)=y_i$ for $i\not=k$.
\end{theorem}
%
%
%
%
%
If we set $y_i=x_i$ for all $i$ in Theorem~\ref{xydiff}, all operators
act like the identity and for $n\ge 1$, we get the following result:
%
%
\begin{cor}
Let $n$ be a positive integer, 
and let $x_i$ ($i=1,2,\dots,2n$) be indeterminates. 
Then we have
\begin{equation}
\Pf_{\lambda}\biggl(x_i-x_j\biggr)_{1\leq i<j\leq2n} 
=(1-\lambda)^{n-1}\prod_{i=1}^{n}(x_{2i-1}-x_{2i}).
\label{eq:xi-xj}
\end{equation}
\end{cor}
\begin{proof}[of Theorem~\ref{xydiff}]
We prove the identity by induction using the recurrence formula for the
$\la$-Pfaffian in Theorem~\ref{th:lambdarec}.

For $n=1$, the identity becomes $x_1-y_2 = x_1 - y_2$.
For $n=2$, the identity becomes $(x_1-y_2)(x_3-y_4) -
\la(x_1-y_3)(x_2-y_4) + \la (x_1-y_4)(x_2-y_3) = (1-\la s_2)
(x_1-y_2)(x_3-y_4)$
 which is trivial to check (and contains an instance of the identity
$(a-c)(b-d) - (a-d)(b-c) = (a-b)(c-d)$ that we will use again below).

Now, we introduce the notation $\tau_k=(1-\la \sigma_k)$ and $e_k=x_k-y_{k+1}$
and we note that the operator $T_k$ commute if their indices do not have
difference 1. Furthermore, the operator $\tau_k$ acts trivially on all
factors different from $e_{k-1}$  and $e_{k+1}$.

For the recurrence, we have to check for $n\ge 3$:

\begin{multline*}
\biggl(\prod_{k=1}^{n-1} \tau_{2k} \prod_{i=1}^n e_{2i-1}\biggr)
\biggl(\prod_{k=2}^{n-2} \tau_{2k} \prod_{i=2}^{n-1} e_{2i-1}\biggr) \\
\hskip-6cm=\biggl(\prod_{k=2}^{n-1} \tau_{2k} \prod_{i=2}^n e_{2i-1}\biggr)
\biggl(\prod_{k=1}^{n-2} \tau_{2k} \prod_{i=1}^{n-1} e_{2i-1}\biggr) \\
-\la \biggl(\prod_{k=1}^{n-2} \tau_{2k+1}  \prod_{i=1}^{n-2} e_{2i}
(x_{2n-2} - y_{2n})\biggr) \biggl(\prod_{k=1}^{n-2} \tau_{2k+1} (x_1 -y_3)
\prod_{i=2}^{n-1} e_{2i}\biggr) \\
+\la \biggl(\prod_{k=1}^{n-2} \tau_{2k+1}  \prod_{i=1}^{n-1} e_{2i} \biggr)
\biggl(\prod_{k=1}^{n-2} \tau_{2k+1} (x_1 -y_3) \prod_{i=2}^{n-2} e_{2i}
(x_{2n-2} -y_{2n})\biggr).
\end{multline*}

But we have
\begin{multline*}
\biggl(\prod_{k=1}^{n-1} \tau_{2k} \prod_{i=1}^n e_{2i-1}\biggr)
\biggl(\prod_{k=2}^{n-2} \tau_{2k} \prod_{i=2}^{n-1} e_{2i-1}\biggr)
=\biggl(\prod_{k=2}^{n-2} \tau_{2k}    \prod_{i=3}^{n-3} e_{2i-1}
\prod_{i=2}^{n-1} e_{2i-1}
\biggl(\tau_2 \tau_{2n-2} e_{1}e_3 e_{2n-3}e_{2n-1}\biggr) \biggr)\\
=\biggl(\prod_{k=2}^{n-1} \tau_{2k} \prod_{i=2}^n e_{2i-1}\biggr)
\biggl(\prod_{k=1}^{n-2} \tau_{2k} \prod_{i=1}^{n-1} e_{2i-1}\biggr).
\end{multline*}

and

\begin{multline*}
\biggl(\prod_{k=1}^{n-1} \tau_{2k+1}  \prod_{i=1}^{n-2} e_{2i}   (x_{2n-2} -
y_{2n})\biggr) \biggl(\prod_{k=1}^{n-1} \tau_{2k+1} (x_1 -y_3)
\prod_{i=2}^{n-1} e_{2i}\biggr) \\
=\biggl(\prod_{k=1}^{n-1} \tau_{2k+1}  \prod_{i=1}^{n-2} e_{2i}   (x_{2n-2}
- y_{2n})  (x_1 -y_3) \prod_{i=2}^{n-1} e_{2i}\biggr) \\
=\biggl(\prod_{k=1}^{n-1} \tau_{2k+1}  \prod_{i=1}^{n-1} e_{2i} \biggr)
\biggl(\prod_{k=1}^{n-1} \tau_{2k+1} (x_1 -y_3) \prod_{i=2}^{n-2} e_{2i}
(x_{2n-2} -y_{2n})\biggr).
\end{multline*}

Therefore, in the recurrence the left-hand side cancels with the first
term on the right-hand side and the second term on the right-hand side
cancels with the third term on the right-hand side.
\end{proof}
%
%
It will be nice to have a combinatorial proof of the above corollary.
\section{Concluding Remarks}

A recent article \cite{Jing-Zhang} proved several identities for the
Pfaffian on the quantum coordinate ring. The coefficient
$(-1)^{cross(m)}q^{\cross(m) + 2 \nest(m)}$ is similar to our
weight. 
However, the matrix elements are non-commuting variables, so
it is not possible to choose special values except for the $q=1$-case
and they give no Dodgson-type formula.

It would be natural to try to introduce two parameters into the recurrence
for the Pfaffian, but we have not found a choice of two parameters that
gives polynomials or Laurent polynomials like the $\lambda$-determinant case in
\cite{RR}, so the question remains open if there is a suitable
deformation of the recurrence.

\end{section}



%
%
%
%

%

\end{document}